\documentclass[preprint,12pt]{elsarticle}

\usepackage{amssymb}
\usepackage{booktabs}
\usepackage{subfigure}
\usepackage{amsmath}
\usepackage[a4paper, margin=2.8cm]{geometry}
\usepackage{amsthm}
\newtheorem{mydef}{Definition}
\newtheorem{mythe}{Theorem}
\newtheorem{myrem}{Remark}
\newtheorem{mypro}{Property}

\journal{ArXiv}

\begin{document}

\begin{frontmatter}



\title{A development of Lagrange interpolation, Part I: Theory}





\author{Mehdi Delkhosh$^{1}$, Kourosh Parand$^{2,3}$, Amir H. Hadian-Rasanan$^{2}$}

\address{$^{1}$Department of Mathematics and Computer Sciences, Islamic Azad University, Bardaskan Branch, Bardaskan, Iran. \\
$^{2}$Department of Computer Sciences, Shahid Beheshti University, G.C., Tehran, Iran.\\
$^{3}$Department of Cognitive Modelling, Institute for Cognitive and Brain Sciences, Shahid Beheshti University, G.C, Tehran, Iran.\\
Emails: mehdidelkhosh@yahoo.com, k\_parand@sbu.ac.ir, amir.h.hadian@gmail.com}

\begin{abstract}
In this work, we introduce the new class of functions which can use to solve the nonlinear/linear multi-dimensional differential equations. Based on these functions, a numerical method is provided which is called the Developed Lagrange Interpolation (DLI). For this, firstly, we define the new class of the functions, called the Developed Lagrange Functions (DLFs), which satisfy in the Kronecker Delta at the collocation points. Then, for the DLFs, the first-order derivative operational matrix of $\textbf{D}^{(1)}$ is obtained, and a recurrence relation is provided to compute the high-order derivative operational matrices of $\textbf{D}^{(m)}$, $m\in \mathbb{N}$; that is, we develop the theorem of the derivative operational matrices of the classical Lagrange polynomials for the DLFs and show that the relation of $\textbf{D}^{(m)}=(\textbf{D}^{(1)})^{m}$ for the DLFs is not established and is developable. Finally, we develop the error analysis of the classical Lagrange interpolation for the developed Lagrange interpolation. 
\end{abstract}


\end{frontmatter}

\section{Introduction}
Many events in medicine, physics, applied sciences, biology, industry, and engineering are implemented by integro-differential/differential/integral equations of various orders. Many of these equations cannot be solved analytically or their analytical solution is costly. For this reason, numerical and semi-analytical methods are used for solving them. As a result, providing a numerical method efficiently and appropriately can be very useful. Many researchers have proposed several methods for solving the equations which have their own disadvantages and advantages, such as the Finite element method (FEM) \cite{refi03}, Finite difference method (FDM) \cite{refi01}, Spectral methods \cite{refpa49,refpa49_2}, Meshfree methods \cite{refpa01}, Adomian decomposition method \cite{refi22}, Fractional spectral collocation method \cite{ref002}, Variational iteration method \cite{refi24}, Homotopy perturbation method \cite{refpa07}, and Exp-function method \cite{refi27}.

There are famous numerical methods such as FDM and FEM which the implementation of them is locally and need to build the network of data, and also the methods such as Meshfree methods that do not require to build the network of data. But, the spectral methods are continuous, global, and do not need to construct the network of data, in addition, in these methods, orthogonal basic functions are usually used for reducing computational costs. Especially in pseudospectral methods, that are an important part of spectral methods, it is usually applied from basis functions where satisfy in the property of Kronecker delta function and the operational metrics which reduce the computational costs. For these reasons, here we are going to introduce a new spectral method which can use for solving some equations in applied sciences.

In recent years, the methods with exponential convergence rate have been introduced by some researchers, such as the hp-spectral element methods of Petrov-Galerkin type \cite{ref023,ref024}, the fractional spectral and pseudo-spectral methods in unbounded domains \cite{ref026,ref025}, the fractional pseudospectral method \cite{ref_FPM}, the generalized pseudospectral method \cite{ref_GPM}, and other methods \cite{ref025_1,ref025_2}. 

In this work, we introduce the method of \textit{Developed Lagrange Interpolation} (DLI) which can use to solve the nonlinear or linear partial/ordinary differential equations, where is a development of the Lagrange interpolation. For this, at first, we introduce the \textit{Developed Lagrange Functions} (DLFs) for the method of DLI and its requirements in Section \ref{sec_DLI}, and we will see that due to the form of defining the functions of DLFs, the DLI has several specific advantages such as the exponential convergence rate. The derivative operational matrices of $\textbf{D}^{(m)}$, $m\in \mathbb{N}$, for DLFs are obtained in Section \ref{sec_deri_matrix}. Section \ref{sec_error} provides the error analysis for DLI. 
A summary and conclusion of the method are given in Section \ref{sec_sum}.

\section{Developed Lagrange Interpolation} \label{sec_DLI}
In this section, firstly, we introduce the new class of functions that called the \textit{developed Lagrange functions}, then some of their properties are expressed, and finally, they are used to provide a method for solving differential equations, where called the \textit{Developed Lagrange Interpolation} (DLI).

\begin{mydef} \label{def1}
Let $f(x)$ be a continuous one-dimensional function on the domain of $\Lambda=[a,b]$, where $a,b \in [-\infty,+\infty]$, and $w(x)$ is a positive weight function on $\Lambda$, then we define $\|f(x)\|_{\infty} = sup\{|f(x)|: x \in \Lambda\}$. And also, for multi-dimensional functions: let $f\big(x^{(1)},x^{(2)},\cdots,x^{(p)}\big)$ be a continuous $p$-dimensional function on the domain of $\Lambda=[a_1,b_1]\times\cdots\times [a_p,b_p]$, where $a_i,b_i \in [-\infty,+\infty]$, and $w(\textit{\textbf{x}})$ is a positive weight function on $\Lambda$, then we define $\|f(\textbf{x})\|_{\infty} = sup\{|f(\textit{\textbf{x}})|: \textit{\textbf{x}} \in \Lambda\}$ where $\textit{\textbf{x}}=\big(x^{(1)},x^{(2)},\cdots,x^{(p)}\big)$.

\end{mydef}

\begin{mydef} \label{def2}
Suppose that $A=[a_0,a_1,\cdots,a_N]^T$ and $B=[b_0,b_1,\cdots,b_M]^T$ are two arbitrary vectors, then the Kronecker product of these two vectors is a $(N+1)(M+1)$-dimensional vector defined as follows:
{\small
$$A\otimes B= \begin{bmatrix} a_0 \\[0.3em]  \vdots \\[0.3em]  a_N \end{bmatrix} \otimes \begin{bmatrix} b_0 \\[0.3em]  \vdots \\[0.3em]  b_M \end{bmatrix}= \begin{bmatrix} a_0\begin{bmatrix} b_0 \\[0.3em]  \vdots \\[0.3em]  b_M \end{bmatrix} \\[0.3em]  \vdots \\[0.3em]  a_N\begin{bmatrix} b_0 \\[0.3em]  \vdots \\[0.3em]  b_M \end{bmatrix} \end{bmatrix}=\begin{bmatrix} a_0 b_0 \\[0.3em]  \vdots \\[0.3em]  a_0b_M  \\[0.3em]  \vdots \\[0.3em]  a_N b_0 \\[0.3em]  \vdots \\[0.3em]  a_Nb_M \end{bmatrix}.$$}
\end{mydef}

\subsection{Developed Lagrange Functions}
Let the points of $\{x_i\}_{i=0}^{N}$ be the arbitrary real values on the domain of $\Lambda$ and $\psi_i(x)$ be the arbitrary functions where are sufficiently differentiable on $\Lambda$. Furthermore, they satisfy in the following two conditions:
\begin{enumerate}
\item[i.] $\psi_i(x_j) \neq \psi_i(x_i)$ for all $i \neq j$.
\item[ii.] ${\psi_i}'(x_i)\neq 0$ for any $i$.
\end{enumerate}

We now introduce the \textit{\textbf{Developed Lagrange Functions}} (DLFs) as a new class of functions for the interpolation methods as follows:
\begin{equation}\label{eq10}
L_j^\psi(x)=L_j(\psi_0,\psi_1,\cdots,\psi_N,x)=\prod_{\substack{i=0 \\ i\neq j}}^N \frac {\psi_i(x)-\psi_i(x_i)}{\psi_i(x_j)-\psi_i(x_i)},~~~~~~~~0 \leq j \leq N.
\end{equation}
Suppose that 
\begin{equation}\label{eq15}
w^\psi(x)=\prod_{i=0}^N (\psi_i(x)-\psi_i(x_i)).
\end{equation}
It is obvious that 
$$\frac{d}{dx}w^\psi(x)\bigg| _{x=x_j}=(w^\psi(x))'\bigg| _{x=x_j}={\psi_j}'(x_j) \prod_{\substack{i=0 \\ i\neq j}}^N (\psi_i(x_j)-\psi_i(x_i)).$$
Thus, we can rewrite Eq. (\ref{eq10}) as:
\begin{equation}\label{eq16}
L_j^\psi(x)=\mu_j\frac {w^\psi(x)}{\psi_j(x)-\psi_j(x_j)},
\end{equation}
where $\mu_j=\frac {{\psi_j}'(x_j)}{(w^\psi(x))'\big| _{x=x_j}}$.

By choosing the different values of $\psi_i(x)$, many new basic functions are produced at different domains, such as:
\begin{enumerate}
\item If $\psi_i(x)=x$ for all $i$, then the classical Lagrange functions are generated.
\item If $\psi_i(x)=x^\delta$ for all $i$, where $\delta$ is a positive real value, then the fractional Lagrange functions are generated.
\item If $\psi_i(x)=\phi(x)$ for all $i$, where $\phi(x)$ is a certain function, then the generalized Lagrange functions are generated.
\item If $\psi_i(x)=\frac{x-L_i}{x+L_i}$ or $\psi_i(x)=\frac{x}{x+L_i}$ for all $i$, where $L_i$ are the positive real values, then the rational Lagrange functions on the semi-infinite domain $[0,\infty)$ are generated.
\item If $\psi_i(x)=e^{ix}$ for all $i$, then the exponential Lagrange functions on the infinite domain $(-\infty,\infty)$ are generated.
\item If $\psi_i(x)=\sin(ix)$ or $\psi_i(x)=\cos(ix)$ for all $i$, then the Fourier Lagrange functions on the infinite domain $(-\infty,\infty)$ are generated.
\item If $\psi_i(x)=e^{ix}$ for any $i=0,...,j$ and $\psi_i(x)=\sin(ix)$ for any $i=j+1,...,N$, then the exponential-Fourier Lagrange functions on the infinite domain $(-\infty,\infty)$ are generated.
\end{enumerate}

\begin{myrem}\label{remark1}
It is necessary to mention that, the $a$ and $b$ values in Definition \ref{def1} are chosen based on the common domain of the $\psi_i(x)$ functions. Furthermore, these values can be infinite.
\end{myrem}

\subsection{Some properties of the DLFs}
Now, we express some properties of the developed Lagrange functions (DLFs).

\begin{mypro}
It is very important that if 
$$\lim_{x\rightarrow \infty}\psi_i(x)=\beta_i < \infty,~~~~~ for~any~i,$$
then we can obtain for any $j$:
$$\lim_{x\rightarrow \infty}L_j^\psi(x)=\prod_{\substack{i=0 \\ i\neq j}}^N \frac {\beta_i-\psi_i(x_i)}{\psi_i(x_j)-\psi_i(x_i)}< \infty.$$
I.e. if all functions of $\psi_i(x)$ are bounded, then all functions of $L_j^\psi(x)$ are also bounded. This feature is very useful and important to solve problems that are defined on the infinite or semi-infinite domains with the help of the Ritz method. And we know that Lagrange polynomials do not have this feature.
\end{mypro}

\begin{mypro}
The Kronecker delta property is available for the DLFs in Eq. (\ref{eq10}), i.e. $L_j^\psi(x_i)=\delta_{ij}$ for all $i,j$. This property reduces computational costs.
\end{mypro}

\begin{mypro}
The property of $\sum_{j=0}^{N}L_j^\psi(x)=1$ is available for the DLFs in Eq. (\ref{eq10}).
\end{mypro}

\begin{mypro}
According to Eq. (\ref{eq10}), $\psi_i(x)$ are the arbitrary functions where are sufficiently differentiable on $\Lambda$ with two simple conditions over them. Therefore, we can choose the various functions of $\psi_i(x)$ in building new classes of basis functions and apply them to solve various integral/differential equations.
\end{mypro}

\begin{mypro}
There is no condition for the existence of the inverse of the functions of $\psi_i(x)$.
\end{mypro}

\subsection{Developed Lagrange Interpolation}

In the classical interpolation, the interpolation methods play a very important role to reduce computational costs. To achieve this goal, the set of the $\{x_i\}_{i=0}^N$ \textit{interpolation points} (IPs), which are distinct, are considered to construct the Lagrange functions. Moreover, in the pseudospectral (collocation) methods, the residual function is set to zero at the $\{y_i\}_{i=0}^N$ \textit{collocation points} (CPs). Generally, the CPs can be distinct from the IPs, but to reduce the costs of the computational, based on the property of the Kronecker delta, they are chosen the same.\\
With the proper choices of $\psi_i(x)$ and making the functions of $L_j^\psi(x)$, we approximate the solution of a problem as follows:
\begin{equation}\label{eq09}
u(x) \approx u_N(x)=\sum_{j=0}^N u_N(x_j)~L_j^\psi(x)=U^T L(x),
\end{equation}
such that
$$u_N(x) \in \mathcal{L}_N^\psi=span\{L_j^\psi(x):0\leq j \leq N,~~x \in \Lambda\},$$
and
\begin{eqnarray}
L(x)&=&[L_0^\psi(x),~L_1^\psi(x),\cdots,~L_N^\psi(x)]^T, \label{eqn_23tagh_2}\\
U&=&[u_N(x_0),u_N(x_1),\cdots,u_N(x_N)]^T,\label{eqn_23tagh_1} 
\end{eqnarray}
where the symbol of $T$ presents the transpose. 

We use it for solving the equations in the following forms:
\begin{enumerate}
\item[\textbf{1.}] \textbf{One-dimensional form:}

We consider the one-dimensional differential equation as follows:
\begin{eqnarray}
\mathcal{Q}^v u(x)&=&h(x), \label{eq06_0}
\end{eqnarray}
with initial and boundary conditions
\begin{eqnarray}
\frac{d^k}{d x^k}u(a)&=&g_k,~~~~~~~k=0,1,\cdots, v_1 -1, \nonumber \\
\frac{d^n}{d x^n}u(b)&=&f_n,~~~~~~~n=0,1,\cdots, v_2 -1, \nonumber
\end{eqnarray}
where $v$ is a non-negative integer; $v_1$ and $v_2$ are constants which $v_1+v_2=v$; $h(x)$ is a known function; $g_k$, and $f_n$ are real constants; $\mathcal{Q}^v$ presents the differential operator which $v$ shows the highest derivative order with respect to $x$; and $x\in [a,b]=\Lambda$.\\

\item[\textbf{2.}] \textbf{Multi-dimensional form:}

We consider the multi-dimensional differential equation as follows:
\begin{eqnarray}
\mathcal{Q}^{v^{(1)},\cdots,v^{(p)}} u\big(x^{(1)},x^{(2)},\cdots,x^{(p)}\big)=h\big(x^{(1)},x^{(2)},\cdots,x^{(p)}\big), ~~~~~ \textit{\textbf{x}}\in\Lambda \label{eq06}
\end{eqnarray}
with initial conditions 
\begin{eqnarray}
\partial^{k_1}_{x^{(1)}} u\big(a_1,x^{(2)},\cdots,x^{(p)}\big)&=&g_{k_1}\big(x^{(2)},\cdots,x^{(p)}\big),~~~~~~~~~k_1=0,\cdots,v^{(1)}_1-1, \nonumber \\
&\vdots&\nonumber \\
\partial^{k_p}_{x^{(p)}} u\big(x^{(1)},\cdots,x^{(p-1)},a_p\big) &=&g_{k_p}\big(x^{(1)},\cdots,x^{(p-1)}\big),~~~~~~k_p=0,\cdots,v^{(p)}_1-1, \nonumber
\end{eqnarray}
and boundary conditions 
\begin{eqnarray}
\partial^{n_1}_{x^{(1)}} u\big(b_1,x^{(2)},\cdots,x^{(p)}\big)&=&f_{n_1}\big(x^{(2)},\cdots,x^{(p)}\big),~~~~~~~~~~n_1=0,\cdots,v^{(1)}_2-1, \nonumber \\
&\vdots&\nonumber \\
\partial^{n_p}_{x^{(p)}} u\big(x^{(1)},\cdots,x^{(p-1)},b_p\big)&=&f_{n_p}\big(x^{(1)},\cdots,x^{(p-1)}\big),~~~~~~~n_p=0,\cdots,v^{(p)}_2-1, \nonumber
\end{eqnarray}
where $\partial^{k_i}_{x^{(i)}}u=\frac{\partial^{k_i}u}{\partial x^{(i)k_i}}$;  $v^{(1)},\cdots,v^{(p)}$ are non-negative integers; $v^{(1)}_1,v^{(1)}_2,\cdots,v^{(p)}_1,v^{(p)}_2$ are constants which $v^{(i)}_1+v^{(i)}_2=v^{(i)}$; $h(\textit{\textbf{x}})$, $g_{k_i}$, and $f_{n_i}$ are the known functions; $\mathcal{Q}^{v^{(1)},\cdots,v^{(p)}}$ presents the differential operator which $v^{(i)}$ shows the highest derivative order with respect to $x^{(i)}$, $i=1,\cdots,p$; and $\textit{\textbf{x}}=\big(x^{(1)},x^{(2)},\cdots,x^{(p)}\big)$.\\
\end{enumerate}

We now have:
\begin{enumerate}
\item[(1)]
For solving Eq. (\ref{eq06_0}), we substitute the solution of $u_N(x)$ in Eq. (\ref{eq09}) into Eq. (\ref{eq06_0}) and then, put the residual function
\begin{eqnarray}
Res(x)=\mathcal{Q}^v u_N(x)-h(x) \label{eq06_02}
\end{eqnarray}
equal to vanish at the $(N-v+1)$ collocation points, and obtain a system of ($N +1$) nonlinear or linear equations as:
\begin{eqnarray}
\mathcal{Q}^v u_N(x_i)-h(x_i)&=&0,~~~~~i=0,1,\cdots,N-v, \nonumber \\
u_N^{(k)}(a)-g_k&=&0,~~~~~k=0,1,\cdots,v_1-1. \label{eq06_04} \\
u_N^{(n)}(b)-f_n&=&0,~~~~~n=0,1,\cdots,v_2-1. \nonumber
\end{eqnarray}

By solving this system, we can obtain the unknown coefficients $u_N(x_j)$ in (\ref{eq09}), and in result, we can calculate the solution of $u_N(x)$.\\

\item[(2)]
For solving Eq. (\ref{eq06}), we define the $p$-dimensional approximate solution of $u_N(\textit{\textbf{x}})$ as follows:
\begin{eqnarray}
u_N\big(x^{(1)},\cdots,x^{(p)}\big) &=& \sum_{i_1=0}^{N_1}\cdots\sum_{i_p=0}^{N_p} u_N\big(x^{(1)}_{i_1},\cdots,x^{(p)}_{i_p}\big) ~L_{i_1}^{\psi_{x^{(1)}}}(x^{(1)})\cdots L_{i_p}^{\psi_{x^{(p)}}}(x^{(p)}) \nonumber \\
&=&U^T\Big(L_{x^{(1)}}(x^{(1)})\otimes\cdots \otimes L_{x^{(p)}}(x^{(p)}))\Big),\label{eq06_22}
\end{eqnarray}
where
\begin{eqnarray}
L_{x^{(i)}}(x^{(i)})&=&\Big[L_0^{\psi_{x^{(i)}}}(x^{(i)}),~L_1^{\psi_{x^{(i)}}}(x^{(i)}),\cdots,~L_{N_i}^{\psi_{x^{(i)}}}(x^{(i)})\Big]^T,~~~~~for~i=1,\cdots,p, \nonumber\\
U &=&\Big[u_N\big(x^{(1)}_{0},x^{(2)}_{0},\cdots,x^{(p-1)}_{0},x^{(p)}_{0}\big),\cdots, u_N\big(x^{(1)}_{0},x^{(2)}_{0},\cdots,x^{(p-1)}_{0},x^{(p)}_{N_p}\big), \nonumber\\
&&~u_N\big(x^{(1)}_{0},x^{(2)}_{1},\cdots,x^{(p-1)}_{0},x^{(p)}_{0}\big),\cdots, u_N\big(x^{(1)}_{0},x^{(2)}_{1},\cdots,x^{(p-1)}_{0},x^{(p)}_{N_p}\big), \nonumber\\
&&~ \vdots \nonumber\\
&&~u_N\big(x^{(1)}_{N_1},x^{(2)}_{N_2},\cdots,x^{(p-1)}_{N_{p-1}},x^{(p)}_{0}\big),\cdots, u_N\big(x^{(1)}_{N_1},x^{(2)}_{N_2},\cdots,x^{(p-1)}_{N_{p-1}},x^{(p)}_{N_p}\big)\Big]^T. \nonumber
\end{eqnarray}

We now substitute the solution of $u_N(\textit{\textbf{x}})$ into Eq. (\ref{eq06}), and then, put the residual function
\begin{eqnarray}
Res(\textit{\textbf{x}})=\mathcal{Q}^{v^{(1)},\cdots,v^{(p)}} u_N(\textit{\textbf{x}})-h(\textit{\textbf{x}}), \label{eq06_23}
\end{eqnarray}
equal to vanish at the $(N_1- v^{(1)}+1 )\cdots(N_p-v^{(p)}+1)$ collocation points, and obtain a system of $(N_1+1)\cdots(N_p+1)$ nonlinear or linear equations as:
\begin{eqnarray}
\mathcal{Q}^{v^{(1)},\cdots,v^{(p)}} u_N\big(x^{(1)}_{i_1},\cdots,x^{(p)}_{i_p}\big)-h\big(x^{(1)}_{i_1},\cdots,x^{(p)}_{i_p}\big)&=&0,\label{eq06_24} \\
\partial^{k_1}_{x^{(1)}} u\big(a_1,x^{(2)}_{i_2},\cdots,x^{(p)}_{i_p}\big) - g_{k_1}\big(x^{(2)}_{i_2},\cdots,x^{(p)}_{i_p}\big)&=&0, ~~~~ k_1=0,\cdots,v^{(1)}_1-1,\nonumber \\
&\vdots&\nonumber \\
\partial^{k_p}_{x^{(p)}} u\big(x^{(1)}_{i_1},\cdots,x^{(p-1)}_{i_{p-1}},a_p\big)  - g_{k_p}\big(x^{(1)}_{i_1},\cdots,x^{(p-1)}_{i_{p-1}}\big) &=&0, ~~~~ k_p=0,\cdots,v^{(p)}_1-1,\nonumber \\
\partial^{n_1}_{x^{(1)}} u\big(b_1,x^{(2)}_{i_2},\cdots,x^{(p)}_{i_p}\big)- f_{n_1}\big(x^{(2)}_{i_2},\cdots,x^{(p)}_{i_p}\big)&=&0, ~~~~ n_1=0,\cdots,v^{(1)}_2-1, \nonumber \\
&\vdots&\nonumber \\
\partial^{n_p}_{x^{(p)}} u\big(x^{(1)}_{i_1},\cdots,x^{(p-1)}_{i_{p-1}},b_p\big) - f_{n_p}\big(x^{(1)}_{i_1},\cdots,x^{(p-1)}_{i_{p-1}}\big) &=&0, ~~~~ n_p=0,\cdots,v^{(p)}_2-1, \nonumber
\end{eqnarray}
where $i_j=0,\cdots,N_j- v^{(j)}$ and $ j=1,\cdots,p$.

By solving this system, we can obtain the unknown coefficients $u_N\big(x^{(1)}_{i_1},\cdots,x^{(p)}_{i_p}\big)$ in (\ref{eq06_22}), and in result, we can calculate the solution of $u_N(\textit{\textbf{x}})$.
\end{enumerate}
The derivatives in Eqs. (\ref{eq06_02}) - (\ref{eq06_24}) can either be calculated directly or calculated using operational matrices. In the next section, In order to reduce the computational costs, we calculate the operational matrices of the derivative for the developed Lagrange functions, which can be applied in the above method. We call this method the \textit{\textbf{Developed Lagrange Interpolation}} (DLI).

\section{Derivative Operational Matrices}\label{sec_deri_matrix}
In this section, at first, we calculate the derivative operational matrix of $\textbf{D}^{(1)}$ for the DLFs, and then, obtain a recursive relation for calculating $\textbf{D}^{(m)}$ for any $m\in \mathbb{N}$.

By using Eq. (\ref{eq09}), we can obtain ($k=0,1,\cdots,N$):
\begin{equation}\label{eq15_1}
\frac{d^m}{dx^m}u_N(x)\bigg| _{x=x_k}=\sum_{j=0}^N u_N(x_j)~\frac{d^m}{dx^m}L_j^\psi(x)\bigg| _{x=x_k} =\sum_{j=0}^N u_N(x_j)~\textbf{D}^{(m)}_{kj},
\end{equation}
i.e.
\begin{equation}\label{eq15_2}
U_N^{(m)}=\textbf{D}^{(m)}U_N,
\end{equation}
where
\begin{eqnarray}
U_N&=&[u_N(x_0),u_N(x_1),\cdots,u_N(x_N)]^T,\nonumber \\
U_N^{(m)}&=&[u_N^{(m)}(x_0),u_N^{(m)}(x_1),\cdots,u_N^{(m)}(x_N)]^T,\nonumber
\end{eqnarray}
and $\textbf{D}^{(m)}=[\textbf{D}^{(m)}_{kj}]$ which $\textbf{D}^{(m)}_{kj}=\frac{d^m}{dx^m}L_j^\psi(x)\Big| _{x=x_k}$.

\subsection{Derivative operational matrix $\textbf{D}^{(1)}$}

We now take the first derivative $\textbf{D}^{(1)}$ of the both sides of Eq. (\ref{eq16}):
$$\frac{d}{dx}L_j^\psi(x)=\mu_j\frac {(w^\psi(x))'(\psi_j(x)-\psi_j(x_j))- {\psi_j}'(x)w^\psi(x)}{(\psi_j(x)-\psi_j(x_j))^2}.$$
Hence, for $k\neq j$, we can obtain:
\begin{equation}\label{eq17}
\textbf{D}^{(1)}_{kj}=\frac{d}{dx}L_j^\psi(x_k)=\frac {(w^\psi(x))'\big| _{x=x_k}}{(w^\psi(x))'\big| _{x=x_j}}\frac {{\psi_j}'(x_j)}{\psi_j(x_k)-\psi_j(x_j)},
\end{equation}
and for $k= j$, by the L'Hopital's rule, we can obtain:
\begin{eqnarray}
\textbf{D}^{(1)}_{jj}&=&\lim_{x\to x_j}\frac{d}{dx}L_j^\psi(x)\overset{Hop.}{=}\lim_{x\to x_j} \mu_j\frac {(w^\psi(x))''(\psi_j(x)-\psi_j(x_j))-{\psi_j}''(x)w^\psi(x)}{2{\psi_j}'(x)(\psi_j(x)-\psi_j(x_j))}\nonumber\\
&\overset{Hop.}{=}&\frac{(w^\psi(x))''\big| _{x=x_j}}{2(w^\psi(x))'\big| _{x=x_j}}-\frac{{\psi_j}''(x_j)}{2{\psi_j}'(x_j)}.
\end{eqnarray}
So, we can obtain the following theorem:
\begin{mythe}\label{theo4}
By using Eqs. (\ref{eq16}) and (\ref{eq09}), and let $\textbf{D}^{(1)}=[\textbf{D}^{(1)}_{kj}]$, where $\textbf{D}^{(1)}_{kj}=\frac{d}{dx}L_j^\psi(x)\Big| _{x=x_k}$, be the derivative operational matrix of the first order in (\ref{eq15_1}). Then:
\begin{equation}\label{eq20}
\textbf{D}^{(1)}_{kj}=\left\{
\begin{array}{l}
{\frac {(w^\psi(x))'\big| _{x=x_k}}{(w^\psi(x))'\big| _{x=x_j}}\frac {{\psi_j}'(x_j)}{\psi_j(x_k)-\psi_j(x_j)},~~~~~~~~~~~~~~~~~k\neq j,}\\
{}\\
{\frac{(w^\psi(x))''\big| _{x=x_j}}{2(w^\psi(x))'\big| _{x=x_j}}-\frac{{\psi_j}''(x_j)}{2{\psi_j}'(x_j)},~~~~~~~~~~~~~~~~~~k=j. }
\end{array}\right.
\end{equation}
where $0 \leq k,j \leq N$. $\square$
\end{mythe}

\subsection{Derivative operational matrix $\textbf{D}^{(m)}$ for any $m\in \mathbb{N}$}

In the book of \cite{ref029} has been proved that the high-order derivative operational matrices for the classical Lagrange polynomials can calculate by the following theorem:
\begin{mythe}\label{theo5}
Suppose that the first-order derivative operational matrix of the classical Lagrange polynomials is exist, then, one can calculate the high-order derivative operational matrices for them as
$$\textbf{D}^{(m)}=\textbf{D}^{(m-1)}\textbf{D}^{(1)}= \underbrace{\textbf{D}^{(1)}\textbf{D}^{(1)}\cdots\textbf{D}^{(1)}}_{m~times}= \big(\textbf{D}^{(1)}\big)^{m},~~~~m\geq2,$$
and
$$\frac{d^m}{dx^m}U_N=\big(\textbf{D}^{(1)}\big)^{m}U_N.$$
\end{mythe}
\begin{proof}
See Ref. \cite{ref029} (page 65). 
\end{proof}
With a simple review, we can see that Theorem \ref{theo5} is not established for the DLFs, for this reason, we now develop it for the DLFs.\\
According to the definition of $L_k^\psi(x)$, we can obtain:
\begin{eqnarray}
&&\textbf{D}^{(1)}_{ik}={\psi_i}'(x_i){L_k^\psi}'(x_i), \label{eq22}\\
&&\textbf{D}^{(2)}_{ik}={\psi_i}'^2(x_i){L_k^\psi}''(x_i)+{\psi_i}''(x_i){L_k^\psi}'(x_i), \label{eq23}\\
&&\textbf{D}^{(3)}_{ik}={\psi_i}'^3(x_i){L_k^\psi}'''(x_i)+3{\psi_i}'(x_i){\psi_i}''(x_i){L_k^\psi}''(x_i)+ {\psi_i}'''(x_i){L_k^\psi}'(x_i). \label{eq24}
\end{eqnarray}
Using Eq. (\ref{eq09}), we have:
$$u_N'(x)=\sum_{j=0}^N u_N(x_j)~\frac{d}{dx}{L_j^\psi}(x).$$
We now set $u_N(x)={L_k^\psi}'(x)$ and then taking $x=x_i$:
$$ {L_k^\psi}''(x_i)=\sum_{j=0}^N {L_k^\psi}'(x_j)~{L_j^\psi}'(x_i),$$
and according to Eqs. (\ref{eq22}) and (\ref{eq23}), we can obtain:
$$\textbf{D}^{(2)}_{ik}={\psi_i}'(x_i)\sum_{j=0}^N \textbf{D}^{(1)}_{ij}\frac{1}{{\psi_j}'(x_j)}\textbf{D}^{(1)}_{jk}+\frac{{\psi_i}''(x_i)}{{\psi_i}'(x_i)}\textbf{D}^{(1)}_{ik},$$
or in the matrix form:
\begin{eqnarray}
\textbf{D}^{(2)}&=&
\left(\begin{array}{ccc}
{\psi_0}'(x_0) & & 0 \\
 & \ddots & \\
0 & & {\psi_N}'(x_N)
\end{array}\right) \times
\left(\begin{array}{ccc}
\textbf{D}^{(1)}_{00} & \cdots & \textbf{D}^{(1)}_{0N} \\
 \vdots& \ddots & \vdots \\
\textbf{D}^{(1)}_{N0} & \cdots & \textbf{D}^{(1)}_{NN}
\end{array}\right)\times
\left(\begin{array}{ccc}
\frac{1}{{\psi_0}'(x_0)} & & 0 \\
 & \ddots & \\
0 & & \frac{1}{{\psi_N}'(x_N)}
\end{array}\right) \nonumber \\
&&~\times
\left(\begin{array}{ccc}
\textbf{D}^{(1)}_{00} & \cdots & \textbf{D}^{(1)}_{0N} \\
 \vdots& \ddots & \vdots \\
\textbf{D}^{(1)}_{N0} & \cdots & \textbf{D}^{(1)}_{NN}
\end{array}\right)+
\left(\begin{array}{ccc}
\frac{{\psi_0}''(x_0)}{{\psi_0}'(x_0)} & & 0 \\
 & \ddots & \\
0 & & \frac{{\psi_N}''(x_N)}{{\psi_N}'(x_N)}
\end{array}\right) \times
\left(\begin{array}{ccc}
\textbf{D}^{(1)}_{00} & \cdots & \textbf{D}^{(1)}_{0N} \\
 \vdots& \ddots & \vdots \\
\textbf{D}^{(1)}_{N0} & \cdots & \textbf{D}^{(1)}_{NN}
\end{array}\right),\nonumber
\end{eqnarray}
i.e.
\begin{eqnarray}
\textbf{D}^{(2)}&=&\textbf{P}\textbf{D}^{(1)}\textbf{P}^{-1}\textbf{D}^{(1)} +\textbf{P}^{(1)}\textbf{P}^{-1}\textbf{D}^{(1)} \nonumber\\
&=&\Big(\textbf{P}\textbf{D}^{(1)}+\textbf{P}^{(1)}\Big)\textbf{P}^{-1}\textbf{D}^{(1)}, \label{eq25}
\end{eqnarray}
where $\textbf{P}=diag({\psi_0}'(x_0),\cdots,{\psi_N}'(x_N))$, and $\textbf{P}^{(1)}$ and $\textbf{P}^{-1}$ are the first-order derivative and inverse of $\textbf{P}$, respectively.

Moreover, according to Eqs. (\ref{eq22}) - (\ref{eq24}), we can obtain:
\begin{eqnarray}
\textbf{D}^{(3)}_{ik}&=&{\psi_i}'^2(x_i)\sum_{j=0}^N \textbf{D}^{(1)}_{ij}\frac{1}{{\psi_j}'(x_j)}\big(\sum_{r=0}^N \textbf{D}^{(1)}_{jr}\frac{1}{{\psi_r}'(x_r)}\textbf{D}^{(1)}_{rk}\big)\nonumber \\
 &&+3{\psi_i}''(x_i)\sum_{j=0}^N \textbf{D}^{(1)}_{ij}\frac{1}{{\psi_j}'(x_j)}\textbf{D}^{(1)}_{jk}+\frac{{\psi_i}'''(x_i)}{{\psi_i}'(x_i)}\textbf{D}^{(1)}_{ik}. \nonumber
\end{eqnarray}
i.e.
\begin{eqnarray}
\textbf{D}^{(3)}&=&\textbf{P}^2\textbf{D}^{(1)}\textbf{P}^{-1}\textbf{D}^{(1)}\textbf{P}^{-1}\textbf{D}^{(1)} +3\textbf{P}^{(1)}\textbf{D}^{(1)}\textbf{P}^{-1}\textbf{D}^{(1)}
+\textbf{P}^{(2)}\textbf{P}^{-1}\textbf{D}^{(1)} \nonumber\\
&=&\Big(\textbf{P}\textbf{D}^{(2)}+2\textbf{P}^{(1)}\textbf{D}^{(1)}+\textbf{P}^{(2)}\Big)\textbf{P}^{-1}\textbf{D}^{(1)}. \label{eq25}
\end{eqnarray}
So, we can obtain the following theorem, where is a development of Theorem \ref{theo5}:

\begin{mythe}\label{theo6}
Suppose that the first-order derivative operational matrix of the DLFs is exist, then, one can calculate the high-order derivative operational matrices $\textbf{D}^{(m)}$, $m\in \mathbb{N}$, for them as
\begin{eqnarray}
\textbf{D}^{(2)}&=&\Big(\textbf{P}\textbf{D}^{(1)}+\textbf{P}^{(1)}\Big)\textbf{P}^{-1}\textbf{D}^{(1)}, \nonumber\\
\textbf{D}^{(3)}&=& \Big(\textbf{P}\textbf{D}^{(2)}+2\textbf{P}^{(1)}\textbf{D}^{(1)}+\textbf{P}^{(2)}\Big)\textbf{P}^{-1}\textbf{D}^{(1)}, \nonumber
\end{eqnarray}
and in general case
\begin{eqnarray}
\textbf{D}^{(m)}=\Bigg(\sum_{k=0}^{m-1}\binom{m-1}{k}\textbf{P}^{(k)}\textbf{D}^{(m-1-k)}\Bigg)\textbf{P}^{-1}\textbf{D}^{(1)},~~~~m\geq2,
\end{eqnarray}
where $\textbf{P}^{(k)}=diag(\psi^{(k+1)}(x_0),\cdots,\psi^{(k+1)}(x_N))$ and $\textbf{D}^{(1)}$ is defined in Theorem \ref{theo4}.
$\square$
\end{mythe}

\begin{myrem}
In the case of $\psi_i(x)=x$, Theorem \ref{theo6} converts to the classical Theorem \ref{theo5}.
\end{myrem}

\begin{myrem}\label{remark5}
It is necessary to mention that in Theorem \ref{theo6}, ${\psi_i}'(x_i)$ should be non-zero for any $i$. I.e., $\psi_i(x)$ and the $x_i$'s points must be selected so that the $\textbf{P}$ matrix be invertible.
\end{myrem}
\begin{myrem}\label{remark8}
In the proper choices of $\psi_i(x)$, the physical conditions of the equation can be important, such as if the equation is defined on the semi-infinite domain and has the algebraic properties then we can choose $\psi_i(x)=\frac{x-L_i}{x+L_i}$ or $\frac{x}{x+L_i}$, and if has the exponential properties then $\psi_i(x)=e^{\pm L_ix}$; if the equation is on the infinite domain then $\psi_i(x)=e^{\pm L_ix}$; and if the equation is on the finite domain then $\psi_i(x)=x^\delta$ or $2(\frac{x}{L_i})^\delta-1$, and etc.
\end{myrem}




\section{Error analysis}\label{sec_error}
In this section, we provide the error analysis to the DLI. We know that the error in the classical Lagrange interpolation can calculate by using the following theorem \cite{TrefethenBook}:
\begin{mythe}\label{theo_Trefethen}
"Let $u(x)$ be analytic in a region $\Omega$ containing distinct points $x_0,\cdots, x_N$; and let $\Gamma$ be a contour in $\Omega$ enclosing these points in the positive direction. The polynomial interpolant $u_N(x)\in\mathbb{P}_N$ to $u(x)$ at $\{x_j\}$ is
\begin{equation}
u_N(x) =\frac{1}{2\pi i}\int_{\Gamma}\frac{u(t)\big(w(t)-w(x)\big)}{w(t)\big(t-x \big)}dt,
\end{equation}
and if $x$ is enclosed by $\Gamma$, the error in the interpolant is
\begin{equation}
u(x)-u_N(x) =\frac{1}{2\pi i}\int_{\Gamma}\frac{w(x)u(t)}{w(t)\big(t-x \big)}dt,
\end{equation}
where $w(x)=\prod_{j=0}^N (x-x_j)$."
\end{mythe}
\begin{proof}
See Theorem 11.1, Page 103 in Ref. \cite{TrefethenBook}.
\end{proof}
We now develop Theorem \ref{theo_Trefethen} for the DLI.

\begin{mythe}\label{theo_erro_analysis}
Let $u(x)$ be analytic in a region $\Omega$ containing the distinct points of $\{\psi_i(x_j)\}_{i,j=0}^{N,N}$ where $x_j$s are also the distinct points, $\psi_i(x_j) \neq \psi_i(x_i)$ for all $i \neq j$, and ${\psi_i}'(x_i)\neq 0$ for any $i$. And let $\Gamma$ be a contour in $\Omega$ enclosing these points in the positive direction. The interpolation function of $u_N(x)$ to $u(x)$ at $\{x_j\}$ is
\begin{equation}
u_N(x)=\frac{1}{N+1}\sum_{j=0}^{N} \frac{1}{2\pi i} \int_{\Gamma}\frac{{\psi_j}'(t)u(t) \big( w^\psi(t) -w^\psi(x)\big)}{w^\psi(t)\big(\psi_j(t)-\psi_j(x) \big)}dt.
\end{equation}
and if $x$ is enclosed by $\Gamma$, the error in the interpolant is
\begin{equation}
u_N(x)-u(x)=\frac{1}{N+1}\sum_{j=0}^{N} \frac{1}{2\pi i} \int_{\Gamma}\frac{{\psi_j}'(t)w^\psi(x)u(t)}{w^\psi(t)\big(\psi_j(x)-\psi_j(t) \big)}dt.
\end{equation}
\end{mythe}

\begin{proof}
Assume that $u_N(x)$ is an interpolation function for $u(x)$ by the DLI in the distinct points of $\{x_j\}$:
\begin{equation}\label{theo_erro_eq01}
u(x) \approx u_N(x)=\sum_{j=0}^N u_N(x_j)~L_j^\psi(x),
\end{equation}
Let $\Gamma_j$ be a contour in the complex $x$-plane that encloses $\psi_j(x_j)$ but it does not include any other points and the point of $\psi_j(x)$.\\
According to the definition of \textit{\textbf{residue}}, the residue of the function ${\psi_j}'(t)/\big(w^\psi(t)(\psi_j(x)-\psi_j(t))\big)$ is equal to ${\psi_j}'(x_j)/\big((w^\psi(x_j))'\big(\psi_j(x)-\psi_j(x_j)\big)\big)$ at the pole $t=x_j$. So, the right side in Eq. (\ref{eq16}) can be written as follows:
\begin{equation}\label{theo_erro_eq02}
\frac {{\psi_j}'(x_j)w^\psi(x)}{(w^\psi(x_j))'\big(\psi_j(x)-\psi_j(x_j)\big)}=\frac{1}{2\pi i} \int_{\Gamma_j}\frac{{\psi_j}'(t)w^\psi(x)}{w^\psi(t)\big(\psi_j(x)-\psi_j(t) \big)}dt,
\end{equation}
By using Eqs. (\ref{eq16}) and (\ref{theo_erro_eq02}), we can write a contour integral form for $L_j^\psi(x)$ as follows:
\begin{equation}\label{theo_erro_eq03}
L_j^\psi(x)=\frac{1}{2\pi i} \int_{\Gamma_j}\frac{{\psi_j}'(t)w^\psi(x)}{w^\psi(t)\big(\psi_j(x)-\psi_j(t) \big)}dt,
\end{equation}
where $\Gamma_j$ encloses $\psi_j(x_j)$. Now, suppose that $\Gamma'$ is a contour that encloses all of the points of $\{\psi_j(x_j)\}$, but still not the point of $\psi_j(x)$, and $u(x)$ is an analytical function interior and on to $\Gamma'$. So, we can now combine these integrals to obtain an expression for the interpolation function of $u_N(x)$ to $u(x)$ in $\{x_j\}$:
\begin{equation}\label{theo_erro_eq04}
u_N(x)=\frac{1}{N+1}\sum_{j=0}^{N} \frac{1}{2\pi i} \int_{\Gamma'} \frac{{\psi_j}'(t)w^\psi(x)u(t)}{w^\psi(t)\big(\psi_j(x)-\psi_j(t) \big)}dt.
\end{equation}
Now, we enlarge the contour of integration to a new contour $\Gamma$ that encloses $\psi_j(x)$ as well as $\{\psi_j(x_j)\}$, and we assume $u(x)$ is analytic interior and on to $\Gamma$. The residue of the integrand of (\ref{theo_erro_eq04}) at $t=x$ is $-u(x)$, because according to the residue definition, we have:
$$Residue_x=\lim_{t \to x}(t-x)\frac{{\psi_j}'(t)w^\psi(x)u(t)}{w^\psi(t)\big(\psi_j(x)-\psi_j(t) \big)}\overset{L'Hopital's~rule}{=}-u(x).$$
So this brings in a new contribution $-u(x)$ to the integral, an equation for the error in the interpolation function obtains:
\begin{equation}\label{theo_erro_eq05}
u_N(x)-u(x)=\frac{1}{N+1}\sum_{j=0}^{N} \frac{1}{2\pi i} \int_{\Gamma}\frac{{\psi_j}'(t)w^\psi(x)u(t)}{w^\psi(t)\big(\psi_j(x)-\psi_j(t) \big)}dt.
\end{equation}
Furthermore, we know that $u(x)$ can be written
$$u(x)=\frac{1}{N+1}\sum_{j=0}^{N} \frac{1}{2\pi i} \int_{\Gamma}\frac{{\psi_j}'(t)w^\psi(t)u(t)}{w^\psi(t)\big(\psi_j(t)-\psi_j(x) \big)}dt,$$
thus by using Eq. (\ref{theo_erro_eq05}), we have
\begin{equation}\label{theo_erro_eq07}
u_N(x)=\frac{1}{N+1}\sum_{j=0}^{N} \frac{1}{2\pi i} \int_{\Gamma}\frac{{\psi_j}'(t)u(t) \big( w^\psi(t) -w^\psi(x)\big)}{w^\psi(t)\big(\psi_j(t)-\psi_j(x) \big)}dt.
\end{equation}
The proof is complete.
\end{proof}

\section{Summary and conclusion} \label{sec_sum}
In this work, an accurate \textit{Developed Lagrange Interpolation} is introduced to solve the nonlinear/linear multi-dimensional differential equations. At first, new functions of the \textit{Developed Lagrange Functions} are defined, where $L_j^\psi(x_i)=\delta_{ij}$ at $x_i$ collocation points. Then, the classical Lagrange theorem is extended for DLFs and the corresponding derivative matrices of $\textbf{D}^{(m)}$ for all $m\in \mathbb{N}$ are calculated. Furthermore, we develop the error analysis of the classical Lagrange interpolation for the developed Lagrange interpolation. As shown in the paper, the present method has many properties, including the Kronecker delta property, implementation is very simple, and by choosing a suitable function for $\psi_i(x)$ we can solve various equations that are alternating or defined in unbounded domains or have conditions in the infinite, and etc. 


\footnotesize
\section*{References}

\end{document}